\documentclass[12pt,oneside,english]{amsart}
\usepackage[T1]{fontenc}
\usepackage[latin9]{inputenc}
\usepackage{geometry}
\usepackage{amsthm}
\usepackage{amstext}
\usepackage{amssymb}

\makeatletter
\numberwithin{equation}{section}
\numberwithin{figure}{section}
\theoremstyle{plain}
\newtheorem{thm}{\protect\theoremname}
  \theoremstyle{plain}
  \newtheorem{prop}[thm]{\protect\propositionname}
  \theoremstyle{plain}
  \newtheorem{lem}[thm]{\protect\lemmaname}
  \theoremstyle{remark}
  \newtheorem{rem}[thm]{\protect\remarkname}

\makeatother

\usepackage{babel}
  \providecommand{\lemmaname}{Lemma}
  \providecommand{\propositionname}{Proposition}
  \providecommand{\remarkname}{Remark}
\providecommand{\theoremname}{Theorem}

\begin{document}

\title{On Improving Roth's Theorem in the Primes}

\author{Eric Naslund}

\date{December 6th, 2013}

\address{Princeton University Mathematics Department, Fine Hall Room 304,
Princeton NJ 08544-1044 }

\email{naslund@math.princeton.edu}
\begin{abstract}
Let $A\subset\left\{ 1,\dots,N\right\} $ be a set of prime numbers
containing no non-trivial arithmetic progressions. Suppose that $A$
has relative density $\alpha=|A|/\pi(N)$, where $\pi(N)$ denotes
the number of primes in the set $\left\{ 1,\dots,N\right\} $. By
modifying Helfgott and De Roton's work \cite{HelfgottRoton}, we improve
their bound and show that 
\[
\alpha\ll\frac{\left(\log\log\log N\right)^{6}}{\log\log N}.
\]

\end{abstract}
\maketitle

\section{Introduction}

In 1936, Erd\"{o}s and Tur\'{a}n \cite{ErdosTuran} conjectured
that if a set $A\subset\mathbb{N}=\left\{ 1,2,3\dots\right\} $ contains
no $k$ term arithmetic progressions, then it cannot be {}``too large.''
We say that $A\subset\mathbb{N}$ has positive (upper) density if
for some $\epsilon>0$ 
\[
\limsup_{x\rightarrow\infty}\frac{1}{x}\sum_{n\leq x}1_{A}(x)\geq\epsilon,
\]
and throughout this section we will exclude those trivial arithmetic
progressions whose difference is $0$. In 1953, Roth \cite{RothTheorem1953}
proved that if a set $A\subset\mathbb{N}=\left\{ 1,2,3\dots\right\} $
contains no non-trivial arithmetic progressions, then $A$ has density
$0$. Quantitatively he showed that any progression free set of integers
$A$ satisfies 
\[
|A\cap\left\{ 1,2,\dots,N\right\} |\ll\frac{N}{\log\log N}.
\]
Roth's Theorem has been improved significantly over the last 60 years
by Heath-Brown, Szemer\'{e}di, Bourgain, \cite{HeathBrown1987,Szemeredi1990roth,Bourgain1999,Bourgain2009}
and most recently Sanders \cite{Sanders2012roth}, who obtained 
\[
|A\cap\left\{ 1,2,\dots,N\right\} |\ll\frac{N\left(\log\log N\right)^{5}}{\log N}.
\]
Moving to the set of prime numbers, which we will denote $\mathcal{P}$,
we define the relative density of a set $A\subset\mathcal{P}$ up
to $N$ to be 
\begin{equation}
\alpha(N)=\frac{|A\cap\left\{ 1,2,\dots,N\right\} |}{|\mathcal{P}\cap\left\{ 1,2,\dots N\right\} |}.\label{eq:relative density}
\end{equation}
In 1939, Van Der Corput \cite{VanDerCorput1939} showed that $\mathcal{P}$
contains infinitely many non-trivial three term arithmetic progressions.
Green \cite{GreenRothinPrimes} proved an analogue of Roth's Theorem
inside the primes, showing that if $A\subset\mathcal{P}$ contains
no non-trivial arithmetic progressions, then 
\[
\alpha(N)\ll\left(\frac{\log\log\log\log\log N}{\log\log\log\log N}\right)^{\frac{1}{2}},
\]
where the notation $f(N)\ll g(N)$ means that there exists an absolute
constant $C$ such that $f(N)\leq Cg(N)$ for all $N\geq1$. Helfgott
and De Roton \cite{HelfgottRoton} improved this density bound, removing
two log's from the denominator to obtain 
\begin{equation}
\alpha(N)\ll\frac{\log\log\log N}{\left(\log\log N\right)^{\frac{1}{3}}}.\label{eq:Helfgott and De Roton density bound}
\end{equation}
Their result implicitly uses the best quantitative bound on Roth's
Theorem in the integers, and when the proof is run through again with
Sander's bound, the density recovered is 
\begin{equation}
\alpha(N)\ll\frac{\left(\log\log\log N\right)^{\frac{5}{2}}}{\left(\log\log N\right)^{\frac{1}{2}}}.\label{eq:Helfgott and De Roton with Sanders}
\end{equation}
Our main result is the following: 
\begin{thm}
\label{thm:Main Theorem}Suppose that $A\subset\mathcal{P}\cap\left[1,N\right]$
has relative density $\alpha$ and contains no non-trivial arithmetic
progressions. Then 
\[
\alpha\ll\frac{\left(\log\log\log N\right)^{6}}{\log\log N}.
\]

\end{thm}
Our proof parallels that of Helfgott and De Roton, and we look at
the convolution of the indicator function of the set of primes and
the indicator function of a set $\Sigma$. We gain a factor of two
in the exponent by using the $L^{2k}$ norm, where $k$ is a slowly
growing function of $N$, rather than the $L^{2}$ norm. Using this
higher norm introduces several combinatorial difficulties which are
dealt with in section \ref{sec:Sieving-the-Primes} and in the proof
of proposition \ref{prop: Key Sieve Lemma}. This $L^{2k}$ norm bound
gives greater control over the outliers, and allows us to choose a
larger subset on which the convolution is uniformly bounded from below.
As in Helfgott and De Roton, the bound on Roth's Theorem yields a
lower bound on the size of the three term progression operator applied
to this uniform set. If the set $\Sigma$ is chosen correctly, the
three term progression operator of the convolution cannot be too far
from that of the indicator function of the primes, which gives the
desired density bound for the primes.

\subsection{Preliminaries and Notation}

For two functions $f,g:\mathbb{N}\rightarrow\mathbb{R}$, we write
$f\ll g$, or $f(x)=O(g(x))$ if there exists a constant $C>0$ such
that $|f(n)|\leq Cg(n)$ for all positive integers $n$. Often we
will look at when $f\ll g$ for sufficiently large $n$, which means
that there exists $N_{0},C>0$ with $|f(n)|\leq Cg(n)$ for all $n\geq N_{0}$. 

To denote $\frac{1}{|S|}\sum_{x\in S}f(x)$, the expectation of $f$
over the set $S$, we write $\mathbb{E}_{x\in S}f(x)$. Given a function
$f:\mathbb{Z}/N\mathbb{Z}\rightarrow\mathbb{C}$, where $N$ is a
prime, we define the Fourier transform to be 
\[
\hat{f}(t)=\mathbb{E}_{x\in\mathbb{Z}/N\mathbb{Z}}f(x)e^{2\pi ixt/N}.
\]
The convolution operation is given by
\begin{equation}
\left(f*g\right)(x)=\mathbb{E}_{y\in\mathbb{Z}/N\mathbb{Z}}f(y)g(x-y),\label{eq:Convo def}
\end{equation}
which is suitably normalized so that 
\begin{equation}
\widehat{f*g}(t)=\hat{f}(t)\hat{g}(t).\label{eq:diagonlized fourier transform}
\end{equation}
The $L^{k}$ and $\ell^{k}$ norms are defined to be 
\[
\|f\|_{L^{k}(\mathbb{Z}/N\mathbb{Z})}=\left(\mathbb{E}_{x\in\mathbb{Z}/N\mathbb{Z}}|f(x)|^{k}\right)^{\frac{1}{k}},
\]
and 
\[
\|\hat{f}\|_{\ell^{k}(\mathbb{Z}/N\mathbb{Z})}=\left(\sum_{x\in\mathbb{Z}/N\mathbb{Z}}|\hat{f}(x)|^{k}\right)^{\frac{1}{k}}.
\]
When there is no ambiguity, we will omit the notation $\ell^{k}(\mathbb{Z}/N\mathbb{Z})$
and $L^{k}(\mathbb{Z}/N\mathbb{Z})$, and simply write $\|\cdot\|_{k}.$
We will make use of the fact that the inner product $\langle f,g\rangle_{L^{2}(\mathbb{Z}/N\mathbb{Z})}=\mathbb{E}_{x\in\mathbb{Z}/N\mathbb{Z}}f(x)\overline{g(x)},$
satisfies Plancherel's identity
\begin{equation}
\mathbb{E}_{x\in\mathbb{Z}/N\mathbb{Z}}f(x)\overline{g(x)}=\sum_{t\in\mathbb{Z}/N\mathbb{Z}}\widehat{f}(t)\overline{\hat{g}(t)},\label{eq: plancherel}
\end{equation}
from which we obtain $\|f\|_{L^{2}(G)}=\|\hat{f}\|_{\ell^{2}(G)}.$
Given functions $f,g,h:\mathbb{Z}/N\mathbb{Z}\rightarrow\mathbb{C}$,
we let $\Lambda(f,g,h)$ denote the three term arithmetic progression
operator defined by 
\[
\Lambda\left(f,g,h\right)=\mathbb{E}_{x,d\in\mathbb{Z}/N\mathbb{Z}}f(x)g(x+d)h(x+2d).
\]
 If $1_{A}$ is the indicator function of a set $A\subset\mathbb{Z}/N\mathbb{Z}$,
then $\Lambda\left(1_{A},1_{A},1_{A}\right)$ counts the total number
of three term progressions in $A$, including the trivial progressions.
For a set $\Sigma\subset\mathbb{Z}/N\mathbb{Z}$, we let $|\Sigma|$
denote the cardinality of $\Sigma$, and $\mu(\Sigma)=\frac{|\Sigma|}{N}$
denote the relative measure.

\section{Sieving the Primes\label{sec:Sieving-the-Primes}}

Let $A\subset[1,N]$ be a subset of the primes with $|A|=\alpha\frac{N}{\log N}$,
and suppose that $\alpha\geq\left(\log N\right)^{-\frac{1}{4}}$.
We will remove the small primes using the {}``$W$ trick,'' which
allows us to effectively apply certain sieve results later on. Let
$W=\prod_{p\leq z}p$ be the product of the primes less than $z$.
Splitting into the different arithmetic progressions modulo $W$,
there will be exactly $\phi(W)$ nontrivial residue classes. By the
pigeon hole principle there exists an arithmetic progression $AP(b)=\left\{ b+nW\ :\ 1\leq n\leq\frac{N}{W}\right\} $
with

\begin{equation}
\left|AP(b)\cap A\right|\geq\alpha\frac{N}{\log N}\frac{1}{\phi(W)}-\frac{W}{\phi(W)}\label{eq: Pigeon hole lower bound arithmetic progression}
\end{equation}
where the $W/\phi(W)$ on the right hand side appears since we are
not including the primes up to $W$. Let $P$ be the least prime larger
than $\frac{3N}{W}$, so that $3\frac{N}{W}<P\leq6\frac{N}{W}$, and
let $A_{0}\subset\left[1,P\right]$ be the set 
\[
A_{0}=\left\{ n=\frac{m-b}{W}:\ m\in AP(b)\cap A\right\} ,
\]
noting that an arithmetic progression in $A_{0}$ can be lifted to
a progression in $A$. Using equation (\ref{eq: Pigeon hole lower bound arithmetic progression})
along with some basic asymptotics for the number of primes, we can
find a lower bound for the size of $A_{0}$. Notice that 
\begin{equation}
\log W=\sum_{p\leq z}\log p=\theta(z)\sim z,\label{eq:theta asymp to z}
\end{equation}
so that $W\approx e^{z}$, and 
\[
\frac{W}{\phi(W)}=\prod_{p\leq z}\left(1-\frac{1}{p}\right)^{-1}\sim e^{\gamma}\log z
\]
where $\gamma=\lim_{x\rightarrow\infty}\left(\sum_{n\leq x}\frac{1}{n}-\log x\right)\approx0.577$
is the Euler-Mascheroni constant. Choosing $z=\frac{1}{4}\log N$,
and $N$ sufficiently large, we may assume that 
\[
\frac{4}{5}z\leq\log W\leq\frac{4}{3}z,
\]
which means that the modulus has size $N^{\frac{1}{5}}\leq W\leq N^{\frac{1}{3}},$
and 
\begin{equation}
\log z\leq\frac{W}{\phi(W)}\leq2\log z.\label{eq:W/phi(W) bound}
\end{equation}
Equation (\ref{eq:W/phi(W) bound}) along with the inequality $\frac{P}{6}\leq\frac{N}{W}$
implies that and $\frac{P\log z}{6}\leq\frac{N}{\phi(W)}$, and so
by (\ref{eq: Pigeon hole lower bound arithmetic progression}) we
have that 
\begin{eqnarray*}
|A_{0}| & \geq & \alpha P\frac{\log z}{6\log N}-2\log z\\
 & \geq & \alpha P\frac{\log z}{10\log N}
\end{eqnarray*}
for $N$ sufficiently large. Each arithmetic progression inside $A_{0}$
corresponds to an arithmetic progression in $A$, and so if $A_{0}$
contains a three term arithmetic progression, then $A$ must as well.
With this in mind, we shift our attention to the progressions inside
$A_{0}$. Define $a=\frac{\log N}{\log z}1_{A_{0}}$ to be the normalized
indicator function for the set $A_{0}$, which is supported on $\left[0,\frac{P}{3}\right]$
since we chose $P>\frac{3N}{W}$. This function satisfies 
\begin{equation}
\|a\|_{1}=\mathbb{E}_{n\in\mathbb{Z}/P\mathbb{Z}}a(n)\geq\frac{\alpha}{10},\label{eq:L^1 norm bound on a}
\end{equation}
and 
\[
\|a\|_{2}^{2}=\mathbb{E}_{n\in\mathbb{Z}/P\mathbb{Z}}a(n)^{2}=\frac{\log N}{\log z}\mathbb{E}_{n\in\mathbb{Z}/P\mathbb{Z}}a(n)=\frac{\log N}{\log z}\|a\|_{1}.
\]
In section \ref{sec:Main-Theorem} we will examine the key quantity
$\Lambda(a,a,a)$ in detail, and show that it cannot be too small
when $\alpha$ is large. To do this, we will need a bound on the $L^{2k}$
norm of the convolution of $a$ and an indicator function $1_{\Sigma}$,
which is discussed in the following section.

\subsection{Bounding the $L^{2k}$ norm of the convolution}

Our goal is to provide bounds on the $L^{2k}$ norm of $a*\sigma$
where $\sigma$ is the indicator function of a set $\Sigma$. 
\begin{prop}
\label{prop: Key Sieve Lemma}Given $z,N,P$ and $a(n)$ as above,let
$k$ be an integer in the range $1\leq k\leq\frac{1}{2}\log^{\frac{1}{3}}z$.
Suppose that $\sigma=\frac{1_{\Sigma}}{\mu(\Sigma)}$ is the normalized
indicator function of a set $\Sigma\subset\left[-\frac{P}{3},\frac{P}{3}\right]\subset\mathbb{Z}/P\mathbb{Z}$.
Then for $N$ greater than some fixed $N_{0}$, we have the bound
\[
\|a*\sigma\|_{2k}\ll k+\left(\frac{\log N}{\log z}\right)^{1-\frac{1}{2k}}|\Sigma|^{-\frac{1}{2k}}.
\]

\end{prop}
We will make use of a theorem of Klimov on the Selberg sieve. Theorem
3 of \cite{Klimov1958SelbergExplicit} states 
\begin{thm}
\label{thm: Klimov's Theorem}(Klimov) Let $1\leq i\leq k$, $1\leq n\leq M$,
$v_{0}\leq v\leq\frac{\sqrt{M}}{\log^{2k}M}$, for a fixed $v_{0}$,
and define $M_{v}\left(q_{i},l_{i}\right)$ to be the number of integers
$n$ for which $p\nmid q_{i}n+l_{i}$, for each $p\leq v$, and each
$1\leq i\leq k$. Then if $u_{0}=O\left(\exp\left(\log^{B}v\right)\right)$
for a fixed constant $B>0$, for $k\geq2$, we have 
\[
M_{v}\left(q_{i},l_{i}\right)\leq\frac{M}{\log^{k}v}k!\prod_{p}\left(1-\frac{\rho(p)}{p}\right)\left(1-\frac{1}{p}\right)^{-k}\left(1+O\left(\frac{\log\log u_{0}}{\log v}\right)\right),
\]
where $u_{0}=\max\left(q_{i},u_{i,j}\right),$ $u_{i,j}=|l_{i}q_{j}-l_{j}q_{i}|$
$\left(1\leq j\leq k\right)$, and $\rho(p)$ is the number of $n\in\left\{ 1,\dots,p\right\} $
such that 
\[
\prod_{i=1}^{k}\left(q_{i}n+l_{i}\right)\equiv0\ \text{mod}\ p.
\]

\end{thm}
We elect to reference the above theorem, rather than theorem 5.7 from
Halberstam and Richert \cite{HalberstamRichert}, as Klimov's result
allows us to make the dependence on the number variables explicit.
As a direct corollary, we have that
\begin{lem}
\label{lem: Klimov redone}Let $k\geq2$, and let $W$ and $b$ satisfy
$\log b\leq2\log P$ and $\log W\leq2\log P$, where $P\geq P_{0}$
for some fixed $P_{0}$. Suppose that we have $k$ pairwise distinct
$b_{i}$, all relatively prime to $W$, such that $b_{i}\leq b$ for
$1\leq i\leq k$, and that $k\leq\frac{\log P}{12\log\log P}$. Then
\[
\left|\left\{ n\leq P:\ b_{1}+nW,\ \dots,\ b_{k}+nW\ \text{all prime}\right\} \right|\ll P\frac{3^{k}k!}{\log^{k}P}\prod_{p}\left(1-\frac{\rho(p)}{p}\right)\left(1-\frac{1}{p}\right)^{-k},
\]
where the constant does not depend on $k$.\end{lem}
\begin{proof}
In Theorem \ref{thm: Klimov's Theorem}, let $q_{i}=W$, $l_{i}=b_{i}$,
$M=P$, $v=P^{1/3}$, and so that $u_{0}\leq2bW$. To apply Klimov's
theorem, we must have $v\leq\frac{\sqrt{P}}{\log^{2k}(P)}$, and $u_{0}=O\left(\exp\left(\log^{B}v\right)\right)$.
The first condition follows from the assumption that $k\leq\frac{1}{12}\frac{\log P}{\log\log P}.$
The upper bound from $u_{0}$ is satisfied with $B=15$ since $\log(2bW)\leq5\log P\leq15\log v$
by the hypothesis of the lemma. It follows that 
\[
\frac{\log\log u_{0}}{\log v}=O\left(\frac{\log\log N}{\log N}\right),
\]
and so 
\[
M_{v}\left(W,b_{i}\right)\ll P\frac{3^{k}k!}{\log^{k}P}\prod_{p}\left(1-\frac{\rho(p)}{p}\right)\left(1-\frac{1}{p}\right)^{-k},
\]
as desired.
\end{proof}
We note that if $k>\frac{\log P}{12\log\log P}$, the term $\frac{3^{k}k!}{\log^{k}P}$
will be large than $1$, and so the bound is weaker than the trivial
upper of $P$. Using this lemma, we prove proposition \ref{prop: Key Sieve Lemma}.
\begin{proof}
By the definition of the $L^{2k}$ norm and the convolution, we have
that 
\[
\|a*\sigma\|_{2k}^{2k}=\mathbb{E}_{x}|\left(a*\sigma\right)(x)|^{2k}=\mathbb{E}_{x}\left|\mathbb{E}_{y}\sigma(y)a(x-y)\right|^{2k}.
\]
Expanding the sum, $\|a*\sigma\|_{2k}^{2k}$ is bounded above by 
\[
\mathbb{E}_{y_{1},\dots,y_{2k}}|\sigma(y_{1})|\cdots|\sigma(y_{2k})|\mathbb{E}_{x}a(x-y_{1})\cdots a(x-y_{2k}).
\]
Since $a$ is supported on $\left[0,\frac{P}{3}\right],$ and $\sigma$
is supported on $\left[-\frac{P}{3},\frac{P}{3}\right],$ there can
be no wrap around inside $\mathbb{Z}/P\mathbb{Z}$, and we have the
upper bound

\begin{equation}
\mathbb{E}_{x}a(x-y_{1})\cdots a(x-y_{2k})\leq\frac{1}{P}\left(\frac{\log N}{\log z}\right)^{2k}\left|\mathcal{A}\left(y_{1},y_{2},\dots,y_{2k}\right)\right|\label{eq:definition of a bound}
\end{equation}
where 
\[
\mathcal{A}\left(y_{1},y_{2},\dots,y_{2k}\right)=\left\{ n\leq P:\ b+\left(n-y_{1}\right)W,\ \dots,\ b+\left(n-y_{2k}\right)W\ \text{all prime}\right\} .
\]
The size of this set of primes is bounded above by the number of $n\leq P$
such that $b+\left(n-y_{1}\right)W$ is prime, and so by the Brun-Titchmarsh
inequality 
\[
\left|\mathcal{A}\left(y_{1},y_{2},\dots,y_{2k}\right)\right|\leq\frac{2PW}{\phi(W)\log\left(P/W\right)},
\]
which we combine with the inequalities $P/W\geq N^{1/3}$ and $\frac{W}{\phi(W)}\leq2\log z$
to obtain
\begin{equation}
\left|\mathcal{A}\left(y_{1},y_{2},\dots,y_{2k}\right)\right|\leq\frac{12P\log z}{\log N}.\label{eq:basic sieve bound}
\end{equation}
However we will need our sieving lemma in the case where each of the
$y_{i}$ is distinct. Let 
\[
I_{l}=\left\{ \left(y_{1},\dots,y_{2k}\right):\ y_{i}\leq\frac{P}{2},\text{ with }l\text{ or less distinct coordinates }y_{i}\right\} ,
\]

\[
J_{l}=\left\{ \left(y_{1},\dots,y_{2k}\right):\ y_{i}\leq\frac{P}{2},\text{ with exactly }l\text{ distinct coordinates }y_{i}\right\} ,
\]
so that we may bound our quantity from above by a sum from $r=1$
to $2k$, and over $\left(y_{1},\dots,y_{2k}\right)\in I_{r}$ or
$J_{r}$. We split this into two cases. When $r<2k$, we will bound
above by the sum over $I_{r}$, and when $r=2k$, we will use a sum
over elements in the set $J_{2k}$. For $r<2k$, we are looking at
\begin{equation}
\frac{1}{P^{2k+1}}\sum_{r=1}^{2k-1}\sum_{\left(y_{1},\dots,y_{2k}\right)\in I_{r}}|\sigma(y_{1})|\cdots|\sigma(y_{2k})|\sum_{x}a(x-y_{1})\cdots a(x-y_{2k})\label{eq: quantity when r<2k}
\end{equation}
which by (\ref{eq:definition of a bound}) and (\ref{eq:basic sieve bound})
is 
\[
\ll\left(\frac{\log N}{\log z}\right)^{2k-1}\sum_{r=1}^{2k-1}\frac{1}{P^{2k}}\sum_{\left(y_{1},\dots,y_{2k}\right)\in I_{r}}|\sigma(y_{1})|\cdots|\sigma(y_{2k})|.
\]
Each term in the sum over $r$ on the right hand side may be bounded
above by $1/|\Sigma|^{2k-r}$ since 
\begin{eqnarray*}
\frac{1}{P^{2k}}\sum_{\left(y_{1},\dots,y_{2k}\right)\in I_{r}}|\sigma(y_{1})|\cdots|\sigma(y_{2k})| & = & \frac{1}{|\Sigma|^{2k-r}}\left(\mathbb{E}_{y_{1},\dots,y_{r}}\sigma\left(y_{1}\right)\cdots\sigma\left(y_{r}\right)\right)\\
 & = & \frac{1}{|\Sigma|^{2k-r}}.
\end{eqnarray*}
If $|\Sigma|=1$, the lemma is trivial, and if $|\Sigma|\geq2$ we
have 
\[
\sum_{r=1}^{2k-1}\frac{1}{|\Sigma|^{2k-r}}\leq\frac{1}{|\Sigma|-1}\leq\frac{2}{|\Sigma|},
\]
which implies that the quantity in (\ref{eq: quantity when r<2k})
is 
\[
\ll\left(\frac{\log N}{\log z}\right)^{2k-1}\frac{2}{|\Sigma|}.
\]
To bound the cardinality of $\mathcal{A}\left(y_{1},\dots,y_{2k}\right)$
for $\left(y_{1},\dots,y_{2k}\right)\in J_{2k}$, we apply Lemma \ref{lem: Klimov redone}.
The conditions of the lemma are satisfied as $b,W\leq N^{1/3}$ and
$k\leq\frac{\log P}{12\log\log P}$ since $k\leq\frac{1}{2}\log^{1/3}z$.
Thus for any $2k$-tuple $\left(y_{1},\dots,y_{2k}\right)\in J_{2k}$,
\begin{equation}
\left|\mathcal{A}\left(y_{1},y_{2},\dots,y_{2k}\right)\right|\ll\frac{3^{2k}(2k)!P}{\left(\log P\right)^{2k}}\prod_{p}\left(1-\frac{\rho(p)}{p}\right)\left(1-\frac{1}{p}\right)^{-2k}.\label{eq:halberstam upper bound quantity}
\end{equation}
On the right hand side of equation (\ref{eq:halberstam upper bound quantity})
$\rho$ is defined so that $\rho(p)=0$ if $p\leq z$, $\rho(p)=2k$
if $p>z$ and $p\nmid(y_{i}-y_{j})$ for all $i\neq j$. The product
over all $p>z$ where $\rho(p)=2k$ is extremely well behaved, and
we have that
\[
\prod_{p>z}\left(1-\frac{2k}{p}\right)\left(1-\frac{1}{p}\right)^{-2k}
\]
is bounded from above and below by absolute constants. To see why,
notice that 
\begin{eqnarray*}
\log\prod_{p>z}\left(1-\frac{2k}{p}\right)\left(1-\frac{1}{p}\right)^{-2k} & = & 2k\sum_{p>z}\sum_{j=1}^{\infty}\frac{1}{j}\left(\frac{1}{p}\right)^{j}-\sum_{p>z}\sum_{j=1}^{\infty}\frac{1}{j}\left(\frac{2k}{p}\right)^{j}\\
 & = & \sum_{p>z}\sum_{j=2}^{\infty}\frac{1}{j}\left(2k\left(\frac{1}{p}\right)^{j}-\left(\frac{2k}{p}\right)^{j}\right),
\end{eqnarray*}
and as $k\leq\frac{1}{2}\log z$, we have the lower bound 
\begin{eqnarray*}
\sum_{p>z}\sum_{j=2}^{\infty}\frac{1}{j}\left(2k\left(\frac{1}{p}\right)^{j}-\left(\frac{2k}{p}\right)^{j}\right) & \geq & -\sum_{j=2}^{\infty}\frac{1}{j}\left(\log z\right)^{j}\sum_{p>z}\frac{1}{p^{j}}\\
 & \geq & -\sum_{j=2}^{\infty}\frac{1}{j}\left(\log z\right)^{j}\sum_{n>z}\frac{1}{n^{j}}\\
 & \gg & -\frac{\log^{2}z}{z},
\end{eqnarray*}
and similarly we have the upper bound 
\[
\sum_{p>z}\sum_{j=2}^{\infty}\frac{1}{j}\left(2k\left(\frac{1}{p}\right)^{j}-\left(\frac{2k}{p}\right)^{j}\right)\leq0
\]
as $2k\left(\frac{1}{p}\right)^{j}\leq\left(\frac{2k}{p}\right)^{j}$
for all $j$. From these bounds, it follows that for a constant independent
of $k$, 
\[
\prod_{p}\left(1-\frac{\rho(p)}{p}\right)\left(1-\frac{1}{p}\right)^{-2k}\ll\left(\log z\right)^{2k}\prod_{\begin{array}{c}
p>z\\
p|y_{i}-y_{j}
\end{array}}\left(1-\frac{\rho(p)}{p}\right)\left(1-\frac{1}{p}\right)^{-2k}.
\]
Since $|y_{i}-y_{j}|\leq P$, and any integer $y\leq P$ can have
at most $\log_{z}P$ prime factors greater than $z$, we see that
there are at most $\left(2k\right)^{2}\frac{\log P}{\log z}$ primes
greater than $z$ which could divide some difference $y_{i}-y_{j}$.
For each $p$ that divides $y_{i}-y_{j}$ for some $i,j$, the worst
case is when$\rho(p)=1$, and we may assume that this is the case
to obtain an upper bound. As $\left(1-\frac{1}{p}\right)^{-1}\leq\frac{z}{z-1}$
for $p>z$, we have 

\[
\prod_{\begin{array}{c}
p>z\\
p|y_{i}-y_{j}
\end{array}}\left(1-\frac{\rho(p)}{p}\right)\left(1-\frac{1}{p}\right)^{-2k}\leq\left(\frac{z}{z-1}\right)^{\left(4k^{2}\right)(2k-1)\frac{\log P}{\log z}}.
\]
Recall that $\log P\leq\log N\leq4z$ since $z=\frac{1}{4}\log N$.
The exponent is bounded above by 
\begin{eqnarray*}
\left(4k^{2}\right)(2k-1)\frac{\log P}{\log z} & \leq & 8k^{3}\frac{\log P}{\log z}\\
 & \leq & \frac{32k^{3}}{\log z}z\\
 & \leq & 4z,
\end{eqnarray*}
where the final inequality follows from the assumption that $k\leq\frac{1}{2}\left(\log z\right)^{\frac{1}{3}}$.
Since 
\[
\left(\frac{z}{z-1}\right)^{4z}\leq5e^{4}
\]
for $z\geq2$, we obtain the inequality 
\[
\prod_{\begin{array}{c}
p>z\\
p|y_{i}-y_{j}
\end{array}}\left(1-\frac{\rho(p)}{p}\right)\left(1-\frac{1}{p}\right)^{-2k}\leq5e^{4},
\]
and so equation (\ref{eq:halberstam upper bound quantity}) becomes
\[
\left|\mathcal{A}\left(y_{1},y_{2},\dots,y_{2k}\right)\right|\leq C3^{2k}(2k)!P\frac{\left(\log z\right)^{2k}}{\left(\log P\right)^{2k}},
\]
for an absolute constant $C$. Thus, for any $\left(y_{1},\dots,y_{2k}\right)\in J_{2k}$,
\[
\mathbb{E}_{x}a(x-y_{1})\cdots a(x-y_{2k})\leq C3^{2k}(2k)!.
\]
Since 
\[
\frac{1}{P^{2k}}\sum_{\left(y_{1},\dots,y_{2k}\right)\in J_{2k}}|\sigma(y_{1})|\cdots|\sigma(y_{2k})|\leq\mathbb{E}_{y_{1},\dots,y_{2k}}|\sigma(y_{1})|\cdots|\sigma(y_{2k})|=1,
\]
the sum over all the $2k$-tuples in $J_{2k}$ is $\leq1$. Combining
the work done so far, we have proven that 
\[
\|a*\sigma\|_{2k}^{2k}\ll3^{2k}(2k)!+\frac{2k}{|\Sigma|}\left(\frac{\log N}{\log z}\right)^{2k-1}.
\]
The stated result then follows from the fact that $\left(n+m\right)^{\frac{1}{2k}}\leq n^{\frac{1}{2k}}+m^{\frac{1}{2k}}$
for $n,m,k\geq1$, and since 
\[
\left(3^{2k}(2k)!\right)^{\frac{1}{2k}}\leq4k,\ \ \text{and}\ \ \left(3k\right)^{\frac{1}{2k}}\leq e^{\frac{3}{2e}}\leq2.
\]

\end{proof}

\section{Main Theorem \label{sec:Main-Theorem}}

Let $a$, $N$, $W$, $P$ and $z$ be defined as in section \ref{sec:Sieving-the-Primes}.
Following \cite{HelfgottRoton}, we define 
\[
R=\text{Spec}_{\delta}\left(a\right)\cup\{1\}=\left\{ x\in\mathbb{Z}/P\mathbb{Z}:\ |\widehat{a}(x)|\geq\delta\right\} \cup\{1\},
\]
and 
\[
B=B\left(R,\epsilon\right)=\left\{ n\in\mathbb{Z}/P\mathbb{Z}:\ \forall x\in R,\ \|\frac{nx}{P}\|\leq\epsilon\right\} ,
\]
where $\|x\|$ denotes the distance from $x$ to the nearest integer.
The set $B$ is called a \emph{Bohr set} with radius $\epsilon$ and
frequency set $R$. Set $\sigma=\frac{1}{\mu(B)}1_{B}$ to be the
normalized indicator function of the Bohr set $B$. By including the
element $1$ in the set $R$, it follows that $\sigma$ will be supported
on $\left[-\frac{P}{4},\frac{P}{4}\right]$ inside $\mathbb{Z}/P\mathbb{Z}$
when $\epsilon<\frac{1}{4}$. Let $h=a*\sigma$ be our prime indicator
smoothed out by the Bohr set $B$. Notice that
\[
\|\sigma\|_{1}=\frac{1}{\mu(B)}\mathbb{E}_{x\in\mathbb{Z}/P\mathbb{Z}}1_{B}=1,
\]
and 
\begin{eqnarray*}
\|h\|_{1} & = & \mathbb{E}_{x\in\mathbb{Z}/P\mathbb{Z}}\mathbb{E}_{y\in\mathbb{Z}/P\mathbb{Z}}\frac{1_{B}(y)}{\mu(B)}a(x-y)\\
 & = & \|\sigma\|_{1}\|a\|_{1},
\end{eqnarray*}
so that by (\ref{eq:L^1 norm bound on a}), which gave the bound $\|a\|_{1}\geq\frac{\alpha}{10}$,
we have 
\begin{equation}
\|h_{1}\|\geq\frac{\alpha}{10}.\label{eq:L^1 bound on h}
\end{equation}
Our goal is to show that there is little difference between the three
term arithmetic progression operator applied to $a$ and $h,$ and
then prove that $\Lambda(h,h,h)$ is large. Let 
\[
\Delta=\left|\Lambda\left(a,a,a\right)-\Lambda\left(h,h,h\right)\right|
\]
where $\Lambda$ is the three term arithmetic progression operator.
In Helfgott and De Roton's paper \cite{HelfgottRoton}, equation (2.6)
on page 7 states that
\begin{lem}
\label{lem: Helfgott and De Roton Lemma} For the above definition
of $\Delta,\epsilon,\delta$ we have 
\[
\Delta\ll\epsilon+\delta^{3/5}.
\]

\end{lem}
The proof of this lemma makes use of Green and Tao's results on the
restriction theory of the Selberg sieve \cite{GreenTaoSelbergSieve}.
Applying proposition \ref{prop: Key Sieve Lemma} with $\sigma=a$,
we find that 
\[
\|a*a\|_{2}\ll1,
\]
and so $\|\hat{a}\|_{4}\ll1$ since $\|a*a\|_{2}^{2}=\|\hat{a}\|_{4}^{4}$
by (\ref{eq: plancherel}) and (\ref{eq:diagonlized fourier transform}).
This yields the bound $|R|\leq C_{4}\delta^{-4}$, on the size of
the dimension of the Bohr set $B$, for an absolute constant $C_{4}$,
as 
\[
|R|\delta^{4}\leq\sum_{t}\left|\hat{a}(t)\right|^{4}\leq C_{4},
\]
 A well known pigeon hole argument tells us that $|B\left(R,\epsilon\right)|\geq N\epsilon^{|R|},$
and so 
\[
\log|B|\geq\log N-|R||\log\epsilon|\geq\log N-C_{4}\delta^{-4}|\log\epsilon|.
\]
We note that an equation nearly identical to the above appears on
page 9 of \cite{HelfgottRoton}. We chose to deduce it again since
the bound on $\|\hat{a}\|_{4}$ was obtained in a different way. From
now on, we will assume that $\epsilon,\delta$ satisfy 
\begin{equation}
C_{4}\delta^{-4}|\log\epsilon|\leq\frac{1}{2}\log N,\label{eq:epsilon delta req}
\end{equation}
so that 
\[
|B|\geq N^{\frac{1}{2}}.
\]
For $k\leq\frac{1}{2}\left(\log z\right)^{\frac{1}{3}}$, proposition
\ref{prop: Key Sieve Lemma} allows us to bound the $\ell^{2k}$ norm
of $h$. Using the inequality 
\[
\frac{1}{2}\log N\geq\left(\log\log N\right)^{\frac{4}{3}},
\]
which holds for all $N\geq2$, along with the fact that 
\[
k\leq\frac{1}{2}\left(\log z\right)^{\frac{1}{3}}\leq\frac{1}{2}\left(\log\log N\right)^{\frac{1}{3}},
\]
we see that 
\[
N^{\frac{1}{2}}\geq\left(\log N\right)^{2k}\geq\left(\frac{\log N}{\log z}\right)^{2k},
\]
and consequently 
\[
|B|\geq N^{\frac{1}{2}}\geq\left(\frac{\log N}{\log z}\right)^{2k}.
\]
Proposition \ref{prop: Key Sieve Lemma} then implies that 
\begin{equation}
\|h\|_{2k}=\|a*\sigma\|_{2k}\ll k.\label{eq:key k norm bound}
\end{equation}
Using this bound $L^{2k}$ norm of $h=a*\sigma$, along with Sanders
bound on Roth's theorem, we are able to show that $\Lambda(h,h,h)$
must be large. 
\begin{prop}
\label{prop: lamba(h,h,h) is large} There exists positive constants
$c_{1},N_{0}>0$ such that if $N\geq N_{0}$, for any $1\leq k\leq\frac{1}{2}\left(\log z\right)^{\frac{1}{3}}$
we have 
\[
\Lambda\left(h,h,h\right)\gg\exp\left(-c_{1}\left(\frac{\alpha}{k}\right)^{-q_{2k}}\left(\log\frac{1}{\alpha}\right)^{5}\right)
\]
where $q_{2k}=\left(1-\frac{1}{2k}\right)^{-1}.$
\end{prop}
To prove this proposition, we will make use of the following lemma
which allows us to find a large subset where the function $h$ is
bounded below uniformly. 
\begin{lem}
\label{lem:Extracting a dense set lemma}Let $q,p>1$ be such that
$\frac{1}{q}+\frac{1}{p}=1$, and let $f:\mathbb{Z}/P\mathbb{Z}\rightarrow\mathbb{R}^{+}$
be a function with $\|f\|_{1}\geq\alpha$, and $\|f\|_{p}\leq C$
for some $C$. Then there exists a subset $L\subset\mathbb{Z}/P\mathbb{Z}$
such for all $n\in L$ $f(n)\geq\frac{\alpha}{2}$, and 
\[
\left(\frac{\alpha}{2C}\right)^{q}\leq\mu(L).
\]
\end{lem}
\begin{proof}
Define $L=\left\{ n:\ f(n)\geq\frac{\alpha}{2}\right\} $, so that
$L$ is the largest possible set satisfying the first condition. Then,
if $1_{L}(n)$ is the indicator function for $L$ we have that 
\[
\alpha\leq\mathbb{E}_{n\in\mathbb{Z}/P\mathbb{Z}}f(n)\leq\frac{\alpha}{2}+\mathbb{E}_{n\in\mathbb{Z}/P\mathbb{Z}}f(n)1_{L}(n).
\]
Applying H\"{o}lder's inequality yields 

\[
\mathbb{E}_{n\in\mathbb{Z}/P\mathbb{Z}}f(n)1_{L}(n)\leq\|1_{L}\|_{q}\|f\|_{p}\leq C\left(\mu(L)\right)^{\frac{1}{q}}
\]
by our assumption that $\|f\|_{p}\leq C$. The lemma then follows
from the resulting inequality $\frac{\alpha}{2}\leq C\left(\mu(L)\right)^{1/q}.$
\end{proof}
Sanders improvement to Roth's theorem \cite{Sanders2012roth} states
that if $\Xi\subset\left[1,N\right]$ with density $\xi$, then 
\begin{equation}
\Lambda\left(1_{\Xi},1_{\Xi},1_{\Xi}\right)\gg\exp\left(-c\xi^{-1}\left(\log\frac{1}{\xi}\right)^{5}\right).\label{eq:Sanders number bound}
\end{equation}
Using this result along with (\ref{eq:key k norm bound}) and lemma
\ref{lem:Extracting a dense set lemma}, we are ready to finish the
proof of proposition \ref{prop: lamba(h,h,h) is large}. 
\begin{proof}
By equation (\ref{eq:key k norm bound}), it follows that $\|h\|_{2k}\ll k$.
Applying lemma \ref{lem:Extracting a dense set lemma} to the function
$h$, we obtain a subset $L\subset\mathbb{Z}/P\mathbb{Z}$ with 
\[
\mu\left(L\right)\gg\left(\frac{\alpha}{k}\right)^{q_{2k}},
\]
and $h(n)\geq\frac{\alpha}{20}$ for all $n\in L$, where $\frac{1}{q_{2k}}=1-\frac{1}{2k}$.
Restricting $h$ to this subset $L$, we obtain the lower bound 
\begin{equation}
\Lambda\left(h,h,h\right)\geq\frac{\alpha^{3}}{20^{3}}\Lambda\left(1_{L},1_{L},1_{L}\right).\label{eq:transfering to progressions in dense set}
\end{equation}
Applying the bound in (\ref{eq:Sanders number bound}) to our set
$L$, we have that 
\[
\Lambda\left(1_{L},1_{L},1_{L}\right)\gg\exp\left(-c\left(\frac{\alpha}{k}\right)^{-q_{2k}}\left(\log\frac{1}{\alpha}\right)^{5}\right)
\]
for some constant $c$, since the density of $L$ is $\gg\left(\frac{\alpha}{k}\right)^{q_{2k}}$.
By equation (\ref{eq:transfering to progressions in dense set}) it
follows that there is a constant $c_{1}$ such that 
\[
\Lambda\left(h,h,h\right)\gg\exp\left(-c_{1}\left(\frac{\alpha}{k}\right)^{-q_{2k}}\left(\log\frac{1}{\alpha}\right)^{5}\right),
\]
as desired.
\end{proof}
Lemma \ref{lem: Helfgott and De Roton Lemma} tells us that 
\[
\left|\Lambda(h,h,h)-\Lambda(a,a,a)\right|\ll\epsilon+\delta^{\frac{3}{5}},
\]
and Proposition \ref{prop: lamba(h,h,h) is large} implies that $\Lambda(h,h,h)$
must be very large. Recalling equation (\ref{eq:epsilon delta req}),
the requirement that $\epsilon,\delta$ satisfy $C_{4}\delta^{-4}|\log\epsilon|\leq\frac{1}{2}\log N$,
we are ready to put everything together and give a precise lower bound
for the size of $\Lambda(a,a,a)$. We now prove Theorem \ref{thm:Main Theorem}.
\begin{proof}
Suppose that $A\subset\mathcal{P}$ contains no nontrivial arithmetic
progressions. Then $A_{0}$ contains only the trivial 3 term arithmetic
progressions, and we have that the three term arithmetic progression
operator is bounded above by
\begin{equation}
\Lambda\left(a,a,a\right)\ll\frac{1}{P}\left(\frac{\log N}{\log z}\right)^{2}.\label{eq:trivial progressions}
\end{equation}
By equation (\ref{eq:trivial progressions}), lemma \ref{lem: Helfgott and De Roton Lemma}
and proposition \ref{prop: lamba(h,h,h) is large} , we must have
\begin{equation}
\frac{1}{P}\left(\frac{\log N}{\log z}\right)^{2}+\epsilon+\delta^{\frac{3}{5}}\gg\exp\left(-c_{1}\left(\frac{\alpha}{k}\right)^{-q_{2k}}\left(\log\frac{1}{\alpha}\right)^{5}\right).\label{eq:final proof inequality 1}
\end{equation}
In the above, for $N$ sufficiently large, the term $\frac{1}{P}\left(\frac{\log N}{\log z}\right)^{2}$
will be negligible compared to the right hand side as we assumed that
$\alpha>\left(\log N\right)^{-1/4}$. Choosing $\epsilon$ and $\delta$
small enough will lead us to a contradiction. In particular, there
exists a fixed postie constant $\eta>0$, independent of $N$ and
$k$, such that choosing 
\[
\epsilon=\delta^{\frac{3}{5}}=\eta\exp\left(-c_{1}\left(\frac{\alpha}{k}\right)^{-q_{2k}}\left(\log\frac{1}{\alpha}\right)^{5}\right),
\]
makes inequality (\ref{eq:final proof inequality 1}) impossible.
These values of $\epsilon,\delta$ satisfy the necessary constraint
$C_{4}\left|\log\epsilon\right|\delta^{-4}\leq\frac{1}{2}\log N$
as long as 
\begin{equation}
\exp\left(c_{2}\left(\frac{\alpha}{k}\right)^{-q_{2k}}\left(\log\frac{1}{\alpha}\right)^{5}\right)\ll\frac{1}{2}\log N\label{eq:final proof inequality 2}
\end{equation}
for some new constant $c_{2}$, where the $\ll$ has consumed $\eta$.
When 
\[
\alpha\geq k\left(2c_{2}\right)^{\frac{1}{q_{2k}}}\frac{\left(\log\log\log N\right)^{\frac{5}{q_{2k}}}}{\left(\log\log N\right)^{\frac{1}{q_{2k}}}},
\]
we have that 
\[
\exp\left(c_{2}\left(\frac{\alpha}{k}\right)^{-q_{2k}}\left(\log\frac{1}{\alpha}\right)^{5}\right)\leq\sqrt{\log N},
\]
and so inequality (\ref{eq:final proof inequality 2}) holds for sufficiently
large $N$. Letting $2k=\left[\log\log\log N\right]$, which satisfies
the necessary bound $k\leq\frac{1}{2}\left(\log z\right)^{\frac{1}{3}}$,
we have that 
\[
\left(\frac{\log\log N}{\log\log\log N}\right)^{\frac{1}{2k}}\sim e,
\]
and so 
\[
k\left(2c_{2}\right)^{\frac{1}{q_{2k}}}\frac{\left(\log\log\log N\right)^{\frac{5}{q_{2k}}}}{\left(\log\log N\right)^{\frac{1}{q_{2k}}}}\geq C\frac{\left(\log\log\log N\right)^{6}}{\log\log N},
\]
for some constant $C$. This means that we will have a contradiction
when $N$ is sufficiently large and the density $\alpha$ satisfies
\[
\alpha\geq C\frac{\left(\log\log\log N\right)^{6}}{\log\log N}
\]
for some absolute constant $C$, which proves the desired result.\end{proof}
\begin{rem}
The proof suggests that we might need the condition $N\geq N_{0}$
for some fixed constant $N_{0}$ in the main theorem. Note however
that this is not necessary given how the result is phrased, as the
constant in the $\ll$ will be so large that it accounts for this.
\end{rem}

\specialsection*{Acknowledgments}

I am very grateful to Julia Wolf for her support, encouragement, and
generous help. I am especially thankful for her patience during our
conversations on the Selberg sieve. I would like to thank Greg Martin
for his helpful comments, and Daniel Fiorilli for leading me to the
paper of Klimov, as well as the anonymous referee for his many useful
suggestions.

\bibliographystyle{plain}

\end{document}